\documentclass{article}

\usepackage{arxiv}

\usepackage[utf8]{inputenc} 
\usepackage[T1]{fontenc}    
\usepackage{hyperref}       
\usepackage{url}            
\usepackage{booktabs}       
\usepackage{amsfonts}       
\usepackage{nicefrac}       
\usepackage{microtype}      
\usepackage{graphicx}
\usepackage{amsmath}
\usepackage{amsthm}
\usepackage{float}
\newtheorem{thm}{Theorem}[section]
\newtheorem{lem}[thm]{Lemma}
\newtheorem{prop}[thm]{Proposition}
\newtheorem{cor}[thm]{Corollary}
\theoremstyle{definition}
\newtheorem{exmp}{Example}[section]
\theoremstyle{remark}
\newtheorem{rem}{Remark}
\theoremstyle{remark}
\newtheorem{ques}{Question}
\theoremstyle{definition}
\newtheorem{defn}{Definition}[section]

\title{Coverings of Configurations, Prime Configurations, and Orbiconfigurations}

\author{
  Benjamin Peet\\
  Department of Mathematics\\
  St. Martin's University\\
  Lacey, WA 98503 \\
  \texttt{bpeet@stmartin.edu} \\
}

\begin{document}
\maketitle

\begin{abstract}

\end{abstract}
This exploratory paper considers the notion of a covering of a configuration and $G$-coverings which are coverings that are quotients under a semi-regular group action. We consider prime configurations, those which cannot $G$-cover other configurations, before considering orbiconfigurations. These are a generalized notion of a configuration in the spirit of an orbifold. We derive some specific results as to when configurations are prime as well as considering when an orbiconfiguration is bad - that is, when it cannot be $G$-covered by a configuration. A number of open questions are posited within.



\section{Introduction}

An incidence geometry is given by a pair $(\mathcal{P},\mathcal{L})$ where $\mathcal{P}=\{p_{1}, \ldots , p_{n}\} $; $\mathcal{L}=\{l_{1}, \ldots ,l_{m}\}$; each $l_{i}\subset \mathcal{P}$; and for any pair $p_{i_{1}},p_{i_{2}}$ there is at most one line $l_{j}$ that contains both elements. Naturally, the elements of $\mathcal{P}$ are known as points and the elements of $\mathcal{L}$ are known as lines (or in some sources as blocks). We make the assumption that the space is connected in the sense that the collection of points cannot be split without splitting a line. An incidence geometry is further known as a configuration if each point is incident with the same number of lines as any other (denoted $s$); each line is incident with the same number of points as any other (denoted $t$); and there are at least three points.

The possible configurations have been considered in particular when $n=m$ and $s=t$. These are known as $(n_{s})$ symmetric configurations. For more details, see in particular Grünbaum's book \textit{Configurations of Points and Lines} \cite{grunbaum2009configurations}. This paper considers more generally $(n_{s},m_{t})$ configurations as described above and the possible semi-regular group actions on these spaces. That is, subgroups of the automorphism group such that each element acts freely on the configuration. In this case, the orbit space may also be a configuration - there are further conditions to ensure this.

Once this has been considered, we turn our attention to orbit spaces that are not themselves configurations and use the modern notion of orbifold to construct a definition of an orbiconfiguration. Orbifolds were introduced by Thurston as a more refined notion of a branched covering. A good source on orbifolds would be Chapter 13 of his notes \textit{The geometry and topology of 3-manifolds} \cite{thurston1979geometry}. This has some connections with (and inspiration due to) the recent work in orbigraphs by Daly et. al. in \textit{Orbigraphs: a graph theoretic analog to Riemannian orbifolds} \cite{daly2019orbigraphs}, but we begin entirely from scratch to define our notion of an orbiconfiguration. 

As per orbifolds (and orbigraphs), some interest lies in which orbiconfigurations cannot be covered by a configuration. We give some examples of such orbiconfigurations.

\section{Preliminaries}

\subsection{Definitions}

For convenience, we will at times suppress $p_{i},l_{j}$ to simply their integers $i,j$. We also use the notation of Betten et. al. in their paper \textit{Counting symmetric $v_{3}$ configurations} \cite{betten2000counting} to denote the lines of certain configurations using modulus. For example $\{1,2,4\}\textit{ mod }7$ would indicate that the lines are $\{1,2,4\},\{2,3,5\},\ldots$ etc.

Due to its' importance to this paper we formally define the notion of an automorphism of a configuration:

\begin{defn}
An \textit{automorphism} on a $(n_{s},m_{t})$-configuration $(\mathcal{P},\mathcal{L})$ is a bijection $f:\mathcal{P}\rightarrow\mathcal{P}$ such that for each $l_{i_{1}}$ there is some $l_{i_{2}}$ such that $f(l_{i_{1}})=l_{i_{2}}$. We use the notation of $Aut((\mathcal{P},\mathcal{L}))$ to refer to the group of all automorphisms of the configuration.
\end{defn}

A \textit{duality} between configurations is an incidence preserving map that sends points to lines and lines to points. A configuration is \textit{self-dual} if there is a duality from itself to itself.

A \textit{Menger graph} is a representation of a configuration as an undirected graph where the points are shown as the vertices and the lines are made up by a collection of edges. See \textit{The topology of the configuration of Desargues and Pappus} by Van Straten \cite{van1947topology} and \textit{Configurations and maps} by Coxeter \cite{coxeter1947configurations}.

As noted in \textit{Planar projective configurations (Part 1)} by Mendelsohn et. al. \cite{mendelsohn1987planar}, a Menger graph does not uniquely determine the configuration. We hence consider the \textit{Levi graph} that was presented by Friedrich Levi in \textit{Finite geometrical systems} \cite{levi1942finite}. This is a bipartite graph where (conventionally) black vertices refer to the points of the configuration and white vertices refer to the lines. Edges exist only between black and white vertices and refer to an incidence of the point and the line. The Levi graph uniquely determines the configuration.

\subsection{Results}

We then state some known results for exposition: 

\begin{lem}
$Aut((\mathcal{P},\mathcal{L}))$ is a subgroup of the symmetric group with order the minimum of $n$ and $m$.
\end{lem}

\begin{proof}
This follows directly from the fact that each automorphism is a both a permutation on the $n$ points and $m$ lines.
\end{proof}

We now make the note that it is well established that $n,m,s,t$ do not uniquely determine a configuration. For example, \cite{betten2000counting} deals with the number of non-isomorphic $(n_{3})$ configurations.

Finally, we state the following two well known foundational results. For more details, see page 15 of \cite{grunbaum2009configurations}.

\begin{prop}
For a $(n_{s},m_{t})$-configuration, the following equation holds:
$$nt=ms$$
\end{prop}

\begin{prop}
For a $(n_{s},m_{t})$-configuration, the following inequalities hold:
$$s(t-1)+1 \leq m $$ $$ t(s-1)+1 \leq n$$
\end{prop}

These will be used as we continue.

\section{Coverings of Configurations}

Coverings of graphs have been considered primarily by Frank Leighton, Dana Angluin, Anthony Gardiner, and Walter Neumann in \cite{leighton1982finite},\cite{angluin1981finite}, and \cite{neumann2009leighton}. We note here again that even though a configuration can be represented by a graph (the Menger graph), a graph can represent two different configurations. Thus the notion of a covering of a configuration is certainly distinct and mention must also be made here that an automorphism of a configuration may not be continuous when viewed as a self-map of the associated graph.

To see this consider:

\begin{exmp}
Take the configuration of 25 points and 10 lines arranged in an array - that is two sets of 5 parallel lines.

There is a well defined automorphism that exchanges two parallel lines ($l_{1}$ and $l_{2}$ say) and leaves all other lines invariant.

\begin{figure}[ht]
\centering
\includegraphics[height=10cm]{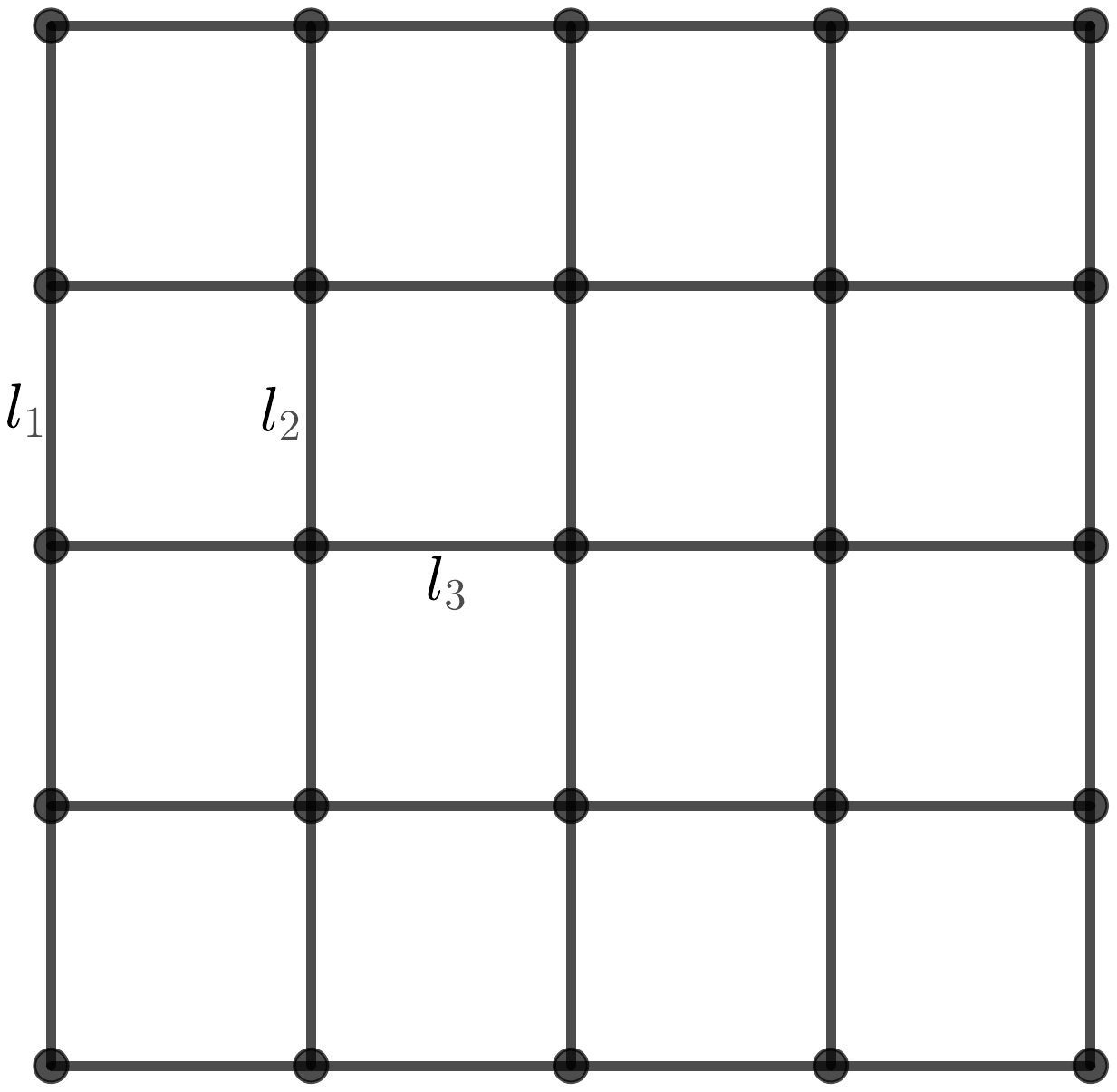}
\caption{Non-continuous automorphism of a Menger graph}
\end{figure}

It can be seen that this automorphism is not a continuous self-map of any Menger graph representation (see Figure 1 as any non-parallel line ($l_{3}$ say) is sent to itself such that 2 points are exchanged and three left invariant. There is no such continuous self-map of the interval or $S^{1}$ that can do this.

\end{exmp}

We then make the following definition:

\begin{defn}
A configuration $(\mathcal{\tilde{P}},\mathcal{\tilde{L}})$ \textit{covers} another configuration $(\mathcal{P},\mathcal{L})$ if there exists an incidence preserving surjection $q:\mathcal{\tilde{P}}\rightarrow \mathcal{P}$ so that each inverse image of a point contains the same number of points and lines are mapped to lines. We call $q:\mathcal{\tilde{P}}\rightarrow \mathcal{P}$ a covering map.
\end{defn}

This again takes some motivation from standard covering space theory in geometric topology. Here by \textit{incidence preserving} we mean that if a point $p$ is incident with a line $l$ in $(\mathcal{\tilde{P}},\mathcal{\tilde{L}})$, then $q(p)$ is incident with $q(l)$ in $(\mathcal{P},\mathcal{L})$ and also that the number of points a line is incident with does not change. This last condition effectively says that both $s$ and $t$ are invariant under the covering. This will become clearer throughout the paper, but one can consider these indices as analogous to dimension in geometric topology, which by necessity is invariant under a covering.

To elucidate this definition we give here an example of such a covering:

\newpage
\begin{exmp}
We choose a $(28_{3},28_{3})$-configuration given by the Menger graph in Figure 2:

\begin{figure}[ht]
\centering
\includegraphics[height=7cm]{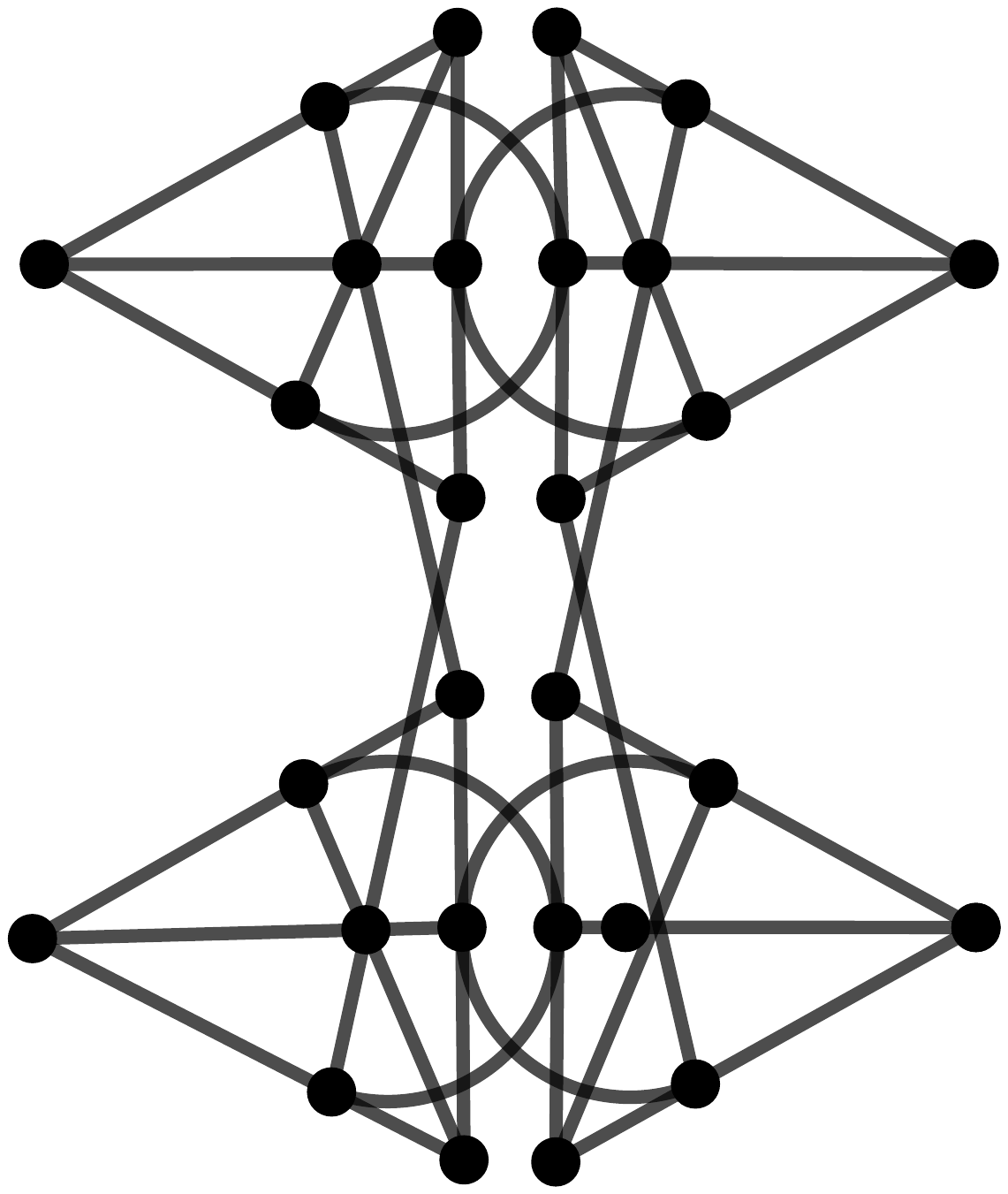}
\caption{Cover of the Fano configuration}
\end{figure}

One can see the vertical and horizontal lines of the symmetry that demonstrate that this configuration covers the Fano configuration (see Figure 4).

\end{exmp}

The above example shows a covering as a quotient of a semi-regular (no fixed points, but non-transitive) action of a group of automorphisms.

We hence consider if all covering maps are in fact quotient maps of a semi-regular group action.

To see that they in fact are not, we first define what we mean by a covering transformation:

\begin{defn}
If $q:\tilde{P} \rightarrow P$ is a covering map where $(\mathcal{\tilde{P}},\mathcal{\tilde{L}})$ covers $(\mathcal{P},\mathcal{L})$, then we say that $g \in Aut((\tilde{P},\tilde{L}))$ is a covering transformation if for any point $x$ we have that
$$q(g(x))=q(x)$$
\end{defn}

The term \textit{deck transformation} is often used by algebraic topologists and could certainly be used interchangeably here. 

It is easy enough to see that for a cover $q:\tilde{P} \rightarrow P$ the set of covering transformations form a group and we call the set of all covering transformations the group of covering transformations under the covering map. We denote this group as $G(q)$. We here show that the elements of this group all act freely both on the set of points and the set of lines, that is:

\begin{prop}
For a covering map $q:\tilde{P} \rightarrow P$ where $(\mathcal{\tilde{P}},\mathcal{\tilde{L}})$ covers $(\mathcal{P},\mathcal{L})$, the group $G(q)$ acts semi-regularly both on the set of points and the set of lines.
\end{prop}

\begin{proof}
We note that by definition of a cover, the values of $s$ and $t$ are invariant. 

Suppose that some element $g \in G(q)$ fixes a line $l$. So then if $g$ does not fix all points of the line, we note that $g(x_{1})=x_{2}$ for $x_{1},x_{2}\in l$ with $x_{1} \neq x_{2}$. So then $q(x_{1})=q(x_{2})$ and hence $q(l)$ has less than $t$ points. This is a contradiction to the definition of a cover.

So now all points on $l$ are fixed by $g$. So taking some point $x$ on the line $l$, if $g$ does not fix all lines incident with $x$ then $q(x)$ is incident with less than $s$ lines. This is again a contradiction.

So then $g$ fixes all lines incident with $x$, and following the above arguments recursively we find that necessarily $g$ fixes all points and all lines, that is, $g$ is the identity.
\end{proof}
Note that the argument could have begun with a point fixed instead.

We now ask whether every covering map is also a quotient map. To see that it is not, we give the following configuration which covers the unique $(4_{3},6_{2})$ configuration:

\begin{exmp}

\begin{figure}[ht]
\centering
\includegraphics[height=9cm]{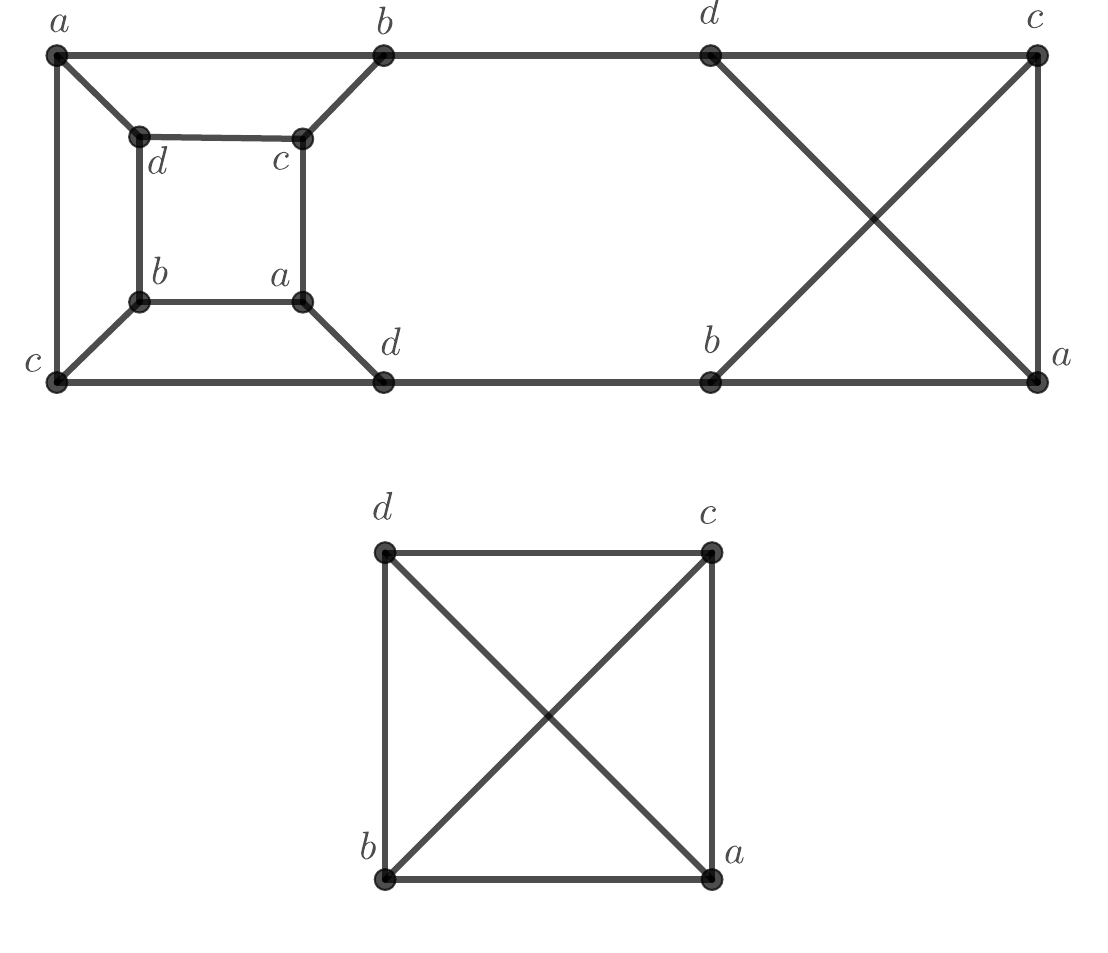}
\caption{Cover of the $(4_{3},6_{2})$ configuration}
\end{figure}

Here the labelling of $a,b,c,d$ demonstrates exactly how this covers the $(4_{3},6_{2})$ configuration.
\end{exmp}

Observing the (lack of) symmetry one can see that this covering map is not a quotient map and indeed, running some GAP code confirms that $G(q)$ is in fact trivial.

We therefore restrict to a definition of a $G$-cover:

\begin{defn}
A cover $q$ of configurations is a $G$-cover if it is also a quotient map under the action of $G(q)$.
\end{defn}

This paper will work exclusively with $G$-covers and hence we can work under the assumption that if $q(x_{1})=q(x_{2})$ then there is some $g \in G(q)$ such that $g(x_{1})=x_{2}$. Here the $x_{i}$ are either points or lines.

We now consider how we can relate automorphisms of the covering configuration to the covered configurations. This is analogous to the lifting criterion in algebraic topology, see for example Allen Hatcher's \textit{Algebraic Topology} \cite{hatcher2005algebraic}.

We first state and prove the following:

\begin{lem}
Suppose that we have a group $G$ of covering transformations for a $G$-covering $q:(\mathcal{\tilde{P}},\mathcal{\tilde{L}}) \rightarrow (\mathcal{P},\mathcal{L})$. Then any automorphism of $(\mathcal{\tilde{P}},\mathcal{\tilde{L}})$ projects to an automorphism of $(\mathcal{P},\mathcal{L})$ if the automorphism commutes with each covering transformation in $G$.
\end{lem}

\begin{proof}
Suppose that $f$ is an automorphism of $(\mathcal{\tilde{P}},\mathcal{\tilde{L}})$ and then consider $qfq^{-1}$. 

To see that this is well defined, suppose that $x_{1},x_{2}$ are distinct points (or lines) of $(\mathcal{\tilde{P}},\mathcal{\tilde{L}})$ with $q(x_{1})=q(x_{2})$.

Then there is a $g \in Aut((\mathcal{P},\mathcal{L}))$ so that $g(x_{1})=x_{2}$. Now note that $qf(x_{2})=qfg(x_{1})=qgf(x_{1})=qf(x_{1})$.

Hence $qfq^{-1}$ is well defined and lifts to $f$.

\end{proof}

We consider now when an automorphism lifts. We first take an example:

\begin{exmp}
We $G$-cover the Fano plane (lines $\{1,2,4\} \textit{ mod }7$) by the configuration given by lines $\{1,2,4\} \textit{ mod }14$. The group of covering transformations is simply $C_{2}$ given by the points modulo $7$.

Now by \cite{betten2000counting}, the automorphism group of the Fano plane is $168$ whilst the size of the automorphism group of our covering configuration is $14$ (in fact $C_{14}$). Hence there are lots of automorphisms that do not lift. We exhibit just the one here.

Using the classic Menger graph representation:

\begin{figure}[ht]
\centering
\includegraphics[height=4cm]{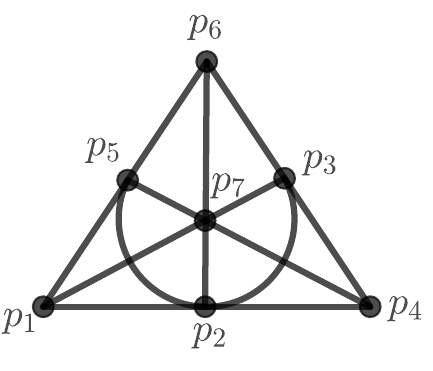}
\caption{Fano configuration}
\end{figure}

We can define a reflection through the vertical line so that $2,6,7$ are fixed; $3$ and $5$ are exchanged; and $1$ and $4$ are exchanged.

For contradiction, suppose that there is a lift of this automorphism. Take the line $\{1,2,4\}$. These points will lift to $\{4\textit{ or } 11,2\textit{ or } 9,1\textit{ or } 8\}$. The only possible combinations which are lines are $\{4,2,1\}$ and  $\{11,9,8\}$. We deal with $\{4,2,1\}$ and note that the second case is very similar.

So $\{3,4,6\}$ lifts to $\{5\textit{ or } 12,1,6\textit{ or } 13\}$. Only $\{12,1,13\}$ is a line. Then $\{4,5,7\}$ lifts to $\{1,3\textit{ or } 10,7\textit{ or } 14\}$. There is no line with these combinations and so this particular automorphism of the Fano configuration cannot lift.

\end{exmp}

We note that in the above example, the only automorphism that does lift is the order $7$ automorphism that rotates the points in their order.

For clarity we give a lemma describing what we require for an automorphism to lift:

\begin{lem}
Suppose that we have a group $G$ of covering transformations for a $G$-covering $q:(\mathcal{\tilde{P}},\mathcal{\tilde{L}}) \rightarrow (\mathcal{P},\mathcal{L})$. Then any automorphism of $(\mathcal{P},\mathcal{L})$ lifts to an automorphism of $(\mathcal{\tilde{P}},\mathcal{\tilde{L}})$ if there is a pointwise lift that sends lines to lines.
\end{lem}

\begin{proof}
So given an automorphism $f:(\mathcal{P},\mathcal{L}) \rightarrow (\mathcal{P},\mathcal{L})$, it is certainly possible to lift to $\tilde{f}$ which is a permutation of the points in $\mathcal{\tilde{P}}$. In order for this to be an automorphism, it must send lines to lines. Hence the result follows.
\end{proof}

Clearly, this is little more than a restating of what is required for an automorphism to lift. We therefore ask the question:

\begin{ques}
For a $G$-covering of configurations, what are the requirements of an automorphism on the covered configuration in order for it to lift? 
\end{ques}

This question is asking for an answer analogous to the Lifting Criterion in algebraic topology. Once again see \cite{hatcher2005algebraic} for more details.

In \cite{angluin1981finite}, \cite{leighton1982finite}, and \cite{neumann2009leighton}, the authors ask when two graphs have a common covering. We then anologously ask:

\begin{ques}
When do two configurations have a common covering? Or have a common $G$-covering?
\end{ques} 

Clearly if $t=2$ then the results of \cite{angluin1981finite}, \cite{leighton1982finite}, and \cite{neumann2009leighton} apply (as a Menger graph uniquely determines the configuration), but for $t\geq3$ the question stands open.

\section{Prime Configurations}

We begin this section with another definition that follows directly from the previous section. This is the notion of something being prime if it cannot be divided in some particular sense.

\begin{defn}
A configuration is \textit{prime} if it does not $G$-cover any other configuration.
\end{defn}

Consider now an action of a subgroup of the automorphism group of a configuration. We determine when this action gives an orbit space which itself can be considered a configuration.

Initially, the orbit space will just be a collection of points (each representing an orbit of points in the configuration) and a collection of lines that contain those points (again, each representing an orbit of lines in the configuration). Hence, we need only consider when this can be given the incidence structure of a configuration.

\begin{prop}
Suppose that $G$ is a subgroup of the automorphism group of $(\mathcal{\tilde{P}},\mathcal{\tilde{L}})$ that acts to give a configuration orbit space $(\mathcal{P},\mathcal{L})$ with the same $s$ and $t$ values. Then the following hold:
\begin{enumerate}
    \item $G$ is a semiregular action on both the collection of points and the collection of lines.
    \item $|G|\leq \frac{n}{3}$
    \item $s(t-1)+1 \leq \frac{m}{|G|}$
    \item $t(s-1)+1 \leq \frac{n}{|G|}$
\end{enumerate}
\end{prop}

\begin{proof}

It is clear that the action must be such that each element is a derangement of both the points and lines as $s$ and $t$ are invariant and using the same argument as the proof of Proposition 3.1. However, there must be more than 2 orbits of points and lines, hence a semi-regular action. 

Then $(\mathcal{P},\mathcal{L})$ has $n_{2}=\frac{n}{|G|}$ points and $m_{2}=\frac{m}{|G|}$ lines. For the orbit space to be a configuration, there must be at least three points, so $|G|\leq \frac{n}{3}$.

Clearly Proposition 2.2 is satisfied, and Proposition 2.3 requires $s(t-1)+1 \leq \frac{m}{|G|}$ and $t(s-1)+1 \leq \frac{n}{|G|}$.
\end{proof}

This proposition yields the corollaries:

\begin{cor}
A configuration is prime if $n$ or $m$ is prime.
\end{cor}

\begin{proof}
$|G|$ divides $n$ (or $m$) so $G$ is trivial or $|G|=n$ (or $|G|=m$), hence no action can satisfy the requirements of Proposition 4.1.
\end{proof}

\begin{cor}
An $(n_{s},m_{t})$ configuration is prime if no factor of $m$ is greater than or equal to $s(t-1)-1$ or no factor of $n$ is greater than or equal to $t(s-1)-1$. 
\end{cor}

\begin{proof}
This follows directly from parts 3. and 4. of Proposition 4.1.
\end{proof}

We now consider some properties of the possible groups $G$.

First note that each element must act freely, and hence each element is a derangement. 

So now we give the following result:

\begin{lem}
Suppose $G$ is a group acting on a $(n_{s},m_{t})$-configuration so that each element except the identity is a derangement. Then any element is a product of $\frac{n}{k}$ disjoint  $k$-cycles.
\end{lem}

\begin{proof}
Suppose that an element $g$ is not a product of $\frac{n}{k}$ disjoint $k$-cycles. So then $g$ can expressed as a product of perhaps an $a$- and a $b$-cycle (and some others too) with $a\neq b$. Then $g^a$ would not be a derangement. This gives our contradiction.
\end{proof}

We now consider the definitions of point- and line- transitivity. A configuration is \textit{point-transitive} (respectively \textit{line-transitive}) if there is only one orbit under the automorphism group acting on the points (respectively lines). If a configuration is not point or line transitive, it is then possible to state the \textit{orbit numbers} of points (and lines). It then follows that: 

\begin{prop}
Suppose that $G$ acts on a configuration $(\mathcal{P},\mathcal{L})$. Suppose that there is some orbit number of points or lines of $(\mathcal{P},\mathcal{L})$ under the action of the automorphism group such that no factor of $|G|$ divides it. Then the orbit space of $(\mathcal{P},\mathcal{L})$ cannot be a configuration.
\end{prop}

\begin{proof}
Take the described number of an orbit of (without loss of generality) points that no factor of $|G|$ divides. Call this collection $\{p_{1},\ldots,p_{k}\}$. So then $G$ acts (non-effectively) on these points. But if no factor of $|G|$ divides $k$ then it must fix some point. Hence the orbit space cannot be a configuration.
\end{proof}

We now give a table with some selected examples $(n_{3})$ from \cite{betten2000counting} to illustrate the above results. Note that $s(t-1)+1=7$, so that by Corollary 4.3, we need $7\leq m=n$.

\begin{table}[H]
\begin{center}
  \begin{tabular}{ | l | l | l |}
    
    \hline
     Number in \cite{betten2000counting}  & Prime/Non-prime  \\ \hline
     7.1 (Fano)   & Prime ($7$ is prime)\\ \hline
     8.1 (Mobius-Kantor)    & Prime (largest factor $4<7$)\\ \hline
     9.1    & Prime (largest factor $3<7$\\ \hline
     9.3  (Pappus)  &  Prime (largest factor $3<7$\\ \hline
     10.8  (Desargues)   & Prime (largest factor $5<7$ \\ \hline
     11.1    & Prime (11 is prime) \\ \hline
     12.1   & Prime (largest factor $6<7$) \\\hline
     13.3   & Prime (13 is prime) \\ \hline
     14.2    & Not prime (covers Fano) \\ \hline
     15.4    & Prime (largest factor $5<7$) \\ \hline
     16.2    & Not Prime (covers Fano) \\ \hline
     18.1    & Not prime (covers 9.1) \\ \hline
     21.1   & Prime (see remark) \\ \hline
  \end{tabular}
  \end{center}
  \caption{\label{tab:Table 1}Table of examples of prime or non-prime configurations from \cite{betten2000counting}}
\end{table}

\begin{rem}
Note that by \cite{betten2000counting}, the configuration $21.1$ has two orbits of $14$ and $7$ lines. For the configuration to be not prime, it would have to have a $C_{3}$ action (which is true - the automorphism group is order 42) and cover $7.1$, the Fano configuration. But by Proposition 4.5, this is not possible. Hence $21.1$ is prime despite there being a group action that satisfies conditions 2-4 of Proposition 4.1.
\end{rem}

\section{Definition of an Orbiconfiguration}

We now open up to consideration of group actions that are not semi-regular and hence the orbit space cannot be given the incidence structure of a configuration.

We begin with a simple motivating example:

\begin{exmp}
Take a cyclic $C_{3}$ action on the 3 point geometry. Then clearly the orbit space is one point with that point covered by 3 points and one line with that line covered by three lines. This is not a configuration as it fails the minimum number of points.
\end{exmp}

We now task ourselves with a rigorous combinatorial definition:

\begin{defn}
An \textit{orbi-incidence structure} is a set of points $\mathcal{P}=\{p_{1}, \ldots , p_{n^{\prime}}\}$  and a set of lines $\mathcal{L}=\{l_{1}, \ldots , l_{m^{\prime}}\}$ with each point and line having an associated positive integer $a(i)$ and $b(j)$.  Each line $l_{j}$ is a set of points with an associated integer $d(j)$ and where each point $p_{i}$ on the line has a further associated non-negative integer $c(i,j)$.
\end{defn}

Here $a(i),b(i)$ refer to the inverse of the "weight" of the point or line. So for instance, if a  point has an associated value of $2$ this would mean the point has weight $\frac{1}{2}$.

The value $c(i,j)$ refers to the \textit{multiplicity of a point $p_{i}$ on the line $l_{j}$}. $d(j)$ refers to the \textit{multiplicity of a line}. 

Note that $c(i,j)=0$ if the point $p_{i}$ is not on the line $l_{j}$.

We then define $n:=\sum\limits_{i=1}^{n^{\prime}}\frac{1}{a(i)}$ and $m:=\sum\limits_{j=1}^{m^{\prime}}\frac{d(j)}{b(j)}$ as the sum of the weights of the points and the lines respectively.

We then consider firstly how many points are incident with a line. This is:

$$t(j)=\sum\limits_{i=1}^{n^{\prime}}d(j)c(i,j)$$

As each line has multiplicity $d(j)$ and contains $\sum\limits_{i=1}^{n^{\prime}}c(i,j)$ points counting multiplicity.

Finally, we consider how many lines are incident with a point. This is:

$$s(i)=\sum\limits_{j=1}^{m^{\prime}}\frac{a(i)c(i,j)d(j)}{b(j)}$$

Of the 4 associated numbers this is the least clear. We explain as follows: Each line has a weight, hence we multiply by $\frac{1}{b(j)}$ to equalize this, then the product $a(i)c(i,j)$ is the number of times the point occurs on a line divided by its' weight.

We can now refine this definition as follows:

\begin{defn}
If any pair of points lies on at most one line (not counting multiplicity), then we refer to the structure as an \textit{orbi-incidence geometry}.
\end{defn}

\begin{defn}
If furthermore $s,t$ are constant functions, then we refer to the structure as an \textit{orbiconfiguration}.
\end{defn}

\section{Properties of Orbiconfigurations}

We first offer an anologue to a Menger graph. Here the points (vertices) and lines (collection of edges) are annotated with the associated integers $a(i)$ and $b(j)$. By convention if $a(i)=1$ or $b(j)=1$ then the integers are not shown. If $d(j)$ is greater than $1$, then the line is drawn $d(j)$ times.

We similarly extend the concept of a Levi graph by having the vertices carry their associated integers $a(i)$ and $b(j)$ (unless equal to $1$); a line of multiplicity $d(j)$ is represented $d(j)$ times; and there may be multiple edges between vertices according to $c(i,j)$.

We now show a few examples of orbiconfigurations with associated Menger and Levi graphs.

\begin{exmp}
Consider the orbiconfiguration with $\mathcal{P}=\{p_{1},p_{2}\},\mathcal{L}=\{\{p_{1},p_{2}\},\{p_{2}\}\}$ and associated values: $a(1)=2,a(2)=1,b(1)=1,b(2)=2,c(1,1)=1,c(1,2)=0,c(2,1)=1,c(2,2)=2,d(1)=1,d(2)=1$.

A Menger graph is:
\begin{figure}[ht]
\centering
\includegraphics[height=2cm]{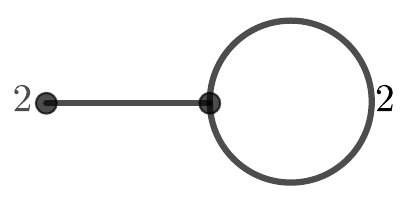}
\caption{Menger graph for orbiconfiguration of Example 6.1}
\end{figure}

A Levi graph is:
\begin{figure}[ht]
\centering
\includegraphics[height=1cm]{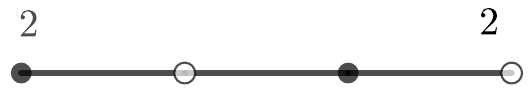}
\caption{Levi graph for orbiconfiguration of Example 6.1}
\end{figure}

So then $n=\frac{1}{2}+\frac{1}{1}=\frac{3}{2}$ and $m=\frac{1}{2}+\frac{1}{1}=\frac{3}{2}$. 

Now $t(1)=1+1=2$ and $t(2)=0+2=2$. 

Finally, $s(1)=\frac{(2)(1)}{1}+\frac{(2)(0)}{2}=2$ and $s(2)=\frac{(1)(1)}{1}+\frac{(1)(2)}{2}=2$. 

Note for future reference that the functions $s$ and $t$ are both equal to the constant $2$.
\end{exmp}

\begin{exmp}
Consider the orbiconfiguration with $\mathcal{P}=\{p_{1},p_{2}\},\mathcal{L}=\{\{p_{1},p_{2}\}\}$ and associated values: $a(1)=1,a(2)=1,b(1)=1,b(2)=1,c(1,1)=1,c(1,2)=1,d(1)=2$.
\newpage
A Menger graph is:
\begin{figure}[h]
\centering
\includegraphics[height=.85cm]{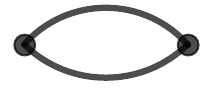}
\caption{Menger graph for orbiconfiguration of Example 6.2}
\end{figure}

A Levi graph is:
\begin{figure}[H]
\centering
\includegraphics[height=2cm]{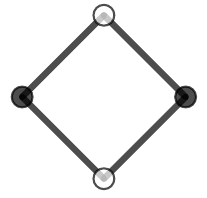}
\caption{Levi graph for orbiconfiguration of Example 6.2}
\end{figure}

\end{exmp}

It should be clear that a Levi graph will uniquely determine an orbiconfiguration whilst a Menger graph will not - consider simply a configuration as a special case of an orbiconfiguration.

We now consider the extension of the concept of a dual configuration:

\begin{defn}
For two orbiconfigurations, a \textit{duality} is an incidence-preserving correspondence that maps points to lines (and vice versa) with the associated integers $a(i)$ and $b(j)$ being carried with the points and lines respectively.
\end{defn}

This definition requires a little more explanation and an example. 

If we let a duality be given by $f:(\mathcal{P}_{1},\mathcal{L}_{1})\rightarrow (\mathcal{P}_{2},\mathcal{L}_{2})$ then each line $f(p_{i})$ will have associated integer $a(i)$ and each point $f(l_{j})$ will have associated integer $b(j)$. Now if a line $l_{j}$ has multiplicity $d(j)$, $f(l_{j})$ will be a collection of $d(j)$ points. If $p_{i}$ and $l_{j}$ are incident $c(i,j)$ times, then $f(l_{j})$ and $f(p_{i})$ will be incident with  $c(i,j)$ times.

The Levi graph makes this much easier to visualize.

\begin{exmp}
We reuse the orbiconfiguration of Example 6.1 to get the dual orbiconfiguration given by: 

$\mathcal{P}=\{p_{1},p_{2}\},\mathcal{L}=\{\{p_{1},p_{2}\},\{p_{2}\}\}$ and associated values: $a(2)=1,a(1)=2,b(1)=1,b(2)=2,c(1,1)=1,c(1,2)=1,c(2,1)=0,c(2,2)=2,d(1)=1,d(2)=1$.

The Levi graph is then given by Figure 6 and this makes it very clear that this orbiconfiguration is self-dual.

\end{exmp}

We now note the following formula still holds:

\begin{prop}
For an orbiconfiguration with values $m,n,s,t$, we have:
$$ns=mt$$
\end{prop}

\begin{proof}
We simply calculate:

$$ns=\sum\limits_{i=1}^{n^{\prime}} \sum\limits_{j=1}^{m^{\prime}}\frac{1}{a(i)}\frac{a(i)c(i,j)d(j)}{b(j)}=\sum\limits_{i=1}^{n^{\prime}}\sum\limits_{j=1}^{m^{\prime}}\frac{c(i,j)d(j)}{b(j)}$$

and:

$$mt=\sum\limits_{j=1}^{m^{\prime}} \sum\limits_{i=1}^{n^{\prime}}\frac{d(j)}{b(j)}c(i,j)=\sum\limits_{j=1}^{m^{\prime}} \sum\limits_{i=1}^{n^{\prime}}\frac{c(i,j)d(j)}{b(j)}$$

Clearly the summands can be exchanged.
\end{proof}

We can also further refine to a \textit{symmetric orbiconfiguration} where $m=n$ and hence also $s=t$.

We now formally show that if a configuration is acted upon by a  group $G$ then the orbit space can be considered as an orbiconfiguration with the values of $s,t$ invariant and the values of $m,n$ divided by the order of the group of covering transformations.

So firstly given a configuration $(\mathcal{P},\mathcal{L})$ with quotient map $q$, the orbit space is initially just a collection of points $\{p_{1}, \ldots , p_{n^{\prime}}\}$. To each of these points we associate values $a(i)$ by:

$$a(i)=\frac{|G|}{\#q^{-1}(p_{i})}$$

This is $|G|$ divided by the orbit number of any lift of $p_{i}$. 

Then the orbit space can also be considered as a collection of lines $\{l_{1}, \ldots , l_{m^{\prime}}\}$. We then associate integers $d(j)$ and $b(j)$ by the following:

$$\frac{b(j)}{d(j)}=\frac{|G|}{\#q^{-1}(l_{j})}$$

This is once again $|G|$ divided by the orbit number of any lift of $l_{j}$. 

Here we assume that the fraction in its lowest terms. Note that either $b(j)$ or $d(j)$ must be equal to one.

Finally, we define:

$$c(i,j)=\#(q^{-1}(p_{i})\cap \tilde{l}_{j})$$

Here $\tilde{l}_{j}$ is any lift of $l_{j}$.

\begin{prop}
Suppose that a configuration $(\mathcal{P},\mathcal{L})$ with values $\tilde{m},\tilde{n},\tilde{s},\tilde{t}$ is acted on by a group $G$. Then the orbit space is an orbiconfiguration when the lines and points are given the associated integers as above. Moreover, $m,n,s,t$ can be determined by:
\begin{enumerate}
    \item $\tilde{s}=s$
    \item $\tilde{t}=t$
    \item $\tilde{m}=|G|m$
    \item $\tilde{n}=|G|n$
\end{enumerate}
\end{prop}

\begin{proof}
We first need to show that for any pair of points, there is at most one line incident with them. Suppose for contradiction that (without loss of generality) $p_{1}$ and $p_{2}$ are incident with both $l_{1}$ and $l_{2}$. Then lifting the lines to (the distinct lines) $\tilde{l}_{1}$ and $\tilde{l}_{2}$, we can lift $p_{1}$ and $p_{2}$ to (distinct) points $\tilde{p}_{1}$ and $\tilde{p}_{2}$ which are on the lines $\tilde{l}_{1}$ and $\tilde{l}_{2}$. This contradicts the fact that $(\mathcal{P},\mathcal{L})$ is a configuration.

We now show that both $s$ and $t$ are constant functions equal to $\tilde{s}$ and $\tilde{t}$.

Note that: $$t(j)=\sum\limits_{i=1}^{n^{\prime}}d(j)c(i,j)=d(j)\sum\limits_{i=1}^{n^{\prime}}\#(q^{-1}(p_{i}) \cap \tilde{l}_{j})$$ If $d(j)=1$, then this is the number of points incident to the line $\tilde{l}_{j}$, which is the constant $\tilde{t}$.

If instead $b(j)=1$, then $d(j)=\frac{\#q^{-1}(l_{j})}{|G|}$, that is $l_{j}$ lifts to $d(j)|G|$ lines. Hence $t(j)$ is the number of points incident to $\tilde{l}_{j}$, which is again the constant $\tilde{t}$

Now: \begin{align*}
s(i)&=\sum\limits_{j=1}^{m^{\prime}}\frac{a(i)c(i,j)d(j)}{b(j)}\\
&=\sum\limits_{j=1}^{m^{\prime}}\frac{|G|}{\#q^{-1}(p_{i})}\#(q^{-1}(p_{i}) \cap \tilde{l}_{j})\frac{\#q^{-1}(l_{j})}{|G|}\\&=\frac{1}{\#q^{-1}(p_{i})}\sum\limits_{j=1}^{m^{\prime}}\#(q^{-1}(p_{i}) \cap \tilde{l}_{j})\#q^{-1}(l_{j})\\
\end{align*} This is the number of lines incident with any point in $q^{-1}(p_{i})$ divided by the number of points in $q^{-1}(p_{i})$, hence the number of lines incident with any point in $q^{-1}(p_{i})$. This is the constant $\tilde{s}$.

Finally, it is quick to see that: $$m=\sum\limits_{j=1}^{m^{\prime}}\frac{b(j)}{d(j)}=\sum\limits_{j=1}^{m^{\prime}}\frac{\#q^{-1}(l_{j})}{|G|}=\frac{\tilde{m}}{|G|}$$ and: $$n=\sum\limits_{i=1}^{n^{\prime}}\frac{1}{a(i)}=\sum\limits_{i=1}^{n^{\prime}}\frac{\#q^{-1}(p_{i})}{|G|}=\frac{\tilde{n}}{|G|}$$
\end{proof}

\section{Good and Bad Orbiconfigurations}

We now make a further definition:

\begin{defn}
An orbiconfiguration is \textit{good} if it can be $G$-covered by a configuration and \textit{bad} if it cannot.
\end{defn}

This uses the familiar terminology of good and bad orbifolds.

We note that as we did not require $s,t$ to be integer and greater than or equal to two, there are some very easily constructed bad orbiconfigurations by noting the invariance of $s$ under actions by Proposition 6.2. 

\begin{cor}
An $(n_{s},m_{t})$ orbiconfiguration is bad if either $s$ or $t$ are not integers or either $s$ or $t$ are less than $2$. 
\end{cor}

We give a brief example to illustrate this:

\begin{exmp}
We take the orbiconfiguration $\mathcal{P}=\{p_{1},p_{2}\},\mathcal{L}=\{\{p_{1},p_{2}\}\}$ and associated values $a(1)=1,a(2)=1,b(1)=2,c(1,1)=1,c(1,2)=1,d(1)=1$.

A Menger graph is given by:
\begin{figure}[h]
\centering
\includegraphics[height=1cm]{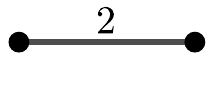}
\caption{Menger graph for orbiconfiguration of Example 7.1}
\end{figure}

Here $s=\frac{1}{2}$ and $t=1$, so by Corollary 7.1, this is a bad orbiconfiguration.

\end{exmp}

Other than these somewhat trivial examples, we have only seen good orbiconfigurations, so we now give an example of a nontrivial bad orbiconfiguration and task ourselves with considering what possible nontrivial bad orbiconfigurations there are in the case when $s=t=2$.

\begin{exmp}
We take the orbiconfiguration given by the Menger graph:
\begin{figure}[ht]
\centering
\includegraphics[height=3cm]{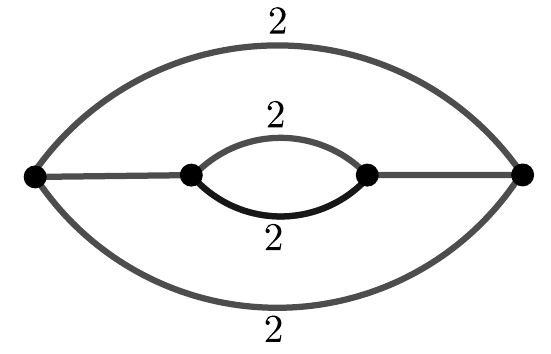}
\caption{Menger graph of an example orbiconfiguration}
\end{figure}

To see that this is necessarily a bad orbiconfiguration we note that $s=t=2$. So that any $G$-covering configuration would necessarily be a polygon with an even number of sides. Yet the automorphism group of such a configuration is $D_{k}$ - the dihedral group on $k$ points. It is fairly immediate to note that any orbit space would either be another polygon configuration or an orbiconfiguration that can be represented by:

\begin{figure}[ht]
\centering
\includegraphics[height=0.75cm]{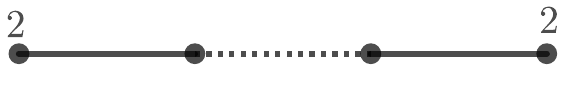}
\caption{Orbit space of an $(n_{2})$ configuration when $n$ is even}
\end{figure}

Or by the dual of the above configuration. Here the dotted line represents a chain of points and lines.

Hence the given orbiconfiguration is bad.

\end{exmp}

We use this example to state the following:
\begin{prop}
An $(n_{2})$ orbiconfiguration is bad if it is not a configuration or one of the following forms:

\begin{figure}[H]
\centering
\includegraphics[height=4cm]{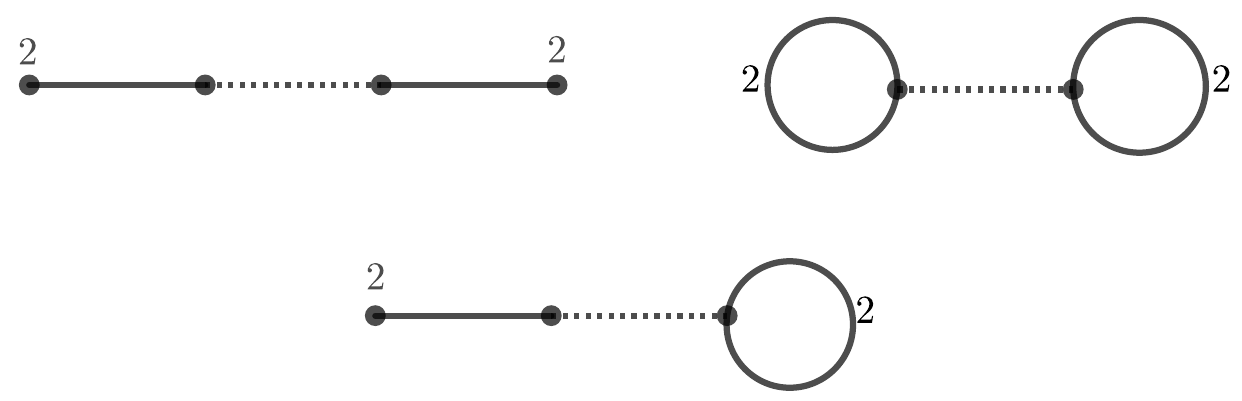}
\caption{Orbit spaces of $(n_{2})$ configurations}
\end{figure}

\end{prop}

\begin{proof}
The proof follows again by just considering how the dihedral group acts on a polygon.
\end{proof}

\begin{ques} What are the bad $(n_{3})$ orbiconfigurations?
\end{ques}

\section{Summary}

We here summarize the work of this paper. We have considered how a covering of a configuration could be defined and refined to a $G$-covering. We then introduced the definition of a configuration being prime if it cannot cover another configuration. Finally we defined an orbiconfiguration - a generalized concept of an orbit space of a configuration. We considered some specific properties of orbiconfigurations and then considered the notion of good/bad orbiconfigurations.

\section*{Acknowledgment}

I would like to thank the anonymous referee for their considered, kind, and very useful critique of the first submission. The paper was greatly improved by addressing the given feedback.

\bibliographystyle{unsrt}  
\bibliography{references} 

\end{document}